\documentclass[a4paper,11pt]{article}
\usepackage[margin=1in]{geometry}
\usepackage{tikz}
\usepackage[small,bf,hang]{caption}
\usepackage{float}

\usepackage{amsmath,amsthm,amssymb}
\usepackage{tikz}
\usepackage{hyperref}
\hypersetup{
    colorlinks=true,
    citecolor = blue,
    linkcolor=black,
    filecolor=magenta,
    urlcolor=cyan
}
  
\newtheorem{theorem}{Theorem}
\newtheorem{question}{Question}

\newtheorem{lemma}{Lemma}
\newtheorem{claim}{Claim}

\title{Reconfiguration of Vertex Colouring and Forbidden Induced Subgraphs}

\author{Manoj Belavadi
 \thanks{Department of Mathematics, Wilfrid Laurier University,
 Waterloo, ON, Canada, N2L 3C5. Email: \texttt{mbelavadi@wlu.ca}. ORCID: 0000-0002-3153-2339. Research supported by the Natural Sciences and Engineering Research Council of Canada (NSERC) grant RGPIN-2016-06517.}
 
\and Kathie Cameron
 \thanks{Department of Mathematics, Wilfrid Laurier University,
 Waterloo, ON, Canada, N2L 3C5. Email: \texttt{kcameron@wlu.ca}. ORCID: 0000-0002-0112-2494. Research supported by the Natural Sciences and Engineering Research Council of Canada (NSERC) grant RGPIN-2016-06517.}
 \and Owen Merkel
 \thanks{Department of Mathematics, Wilfrid Laurier University,
 Waterloo, ON, Canada, N2L 3C5. Email: \texttt{owenmerkel@gmail.com}. ORCID: 0000-0002-8839-6243. Research supported by the Natural Sciences and Engineering Research Council of Canada (NSERC) grant RGPIN-2016-06517.}}

\begin{document}
\maketitle


\begin{abstract}
The reconfiguration graph of the $k$-colourings, denoted $\mathcal{R}_k(G)$, is the graph whose vertices are the $k$-colourings of $G$ and two colourings are adjacent in $\mathcal{R}_k(G)$ if they differ in colour on exactly one vertex. In this paper, we investigate the connectivity and diameter of $\mathcal{R}_{k+1}(G)$ for a $k$-colourable graph $G$ restricted by forbidden induced subgraphs. We show that $\mathcal{R}_{k+1}(G)$ is connected for every $k$-colourable $H$-free graph $G$ if and only if $H$ is an induced subgraph of $P_4$ or $P_3+P_1$. We also start an investigation into this problem for classes of graphs defined by two forbidden induced subgraphs. We show that if $G$ is a $k$-colourable ($2K_2$, $C_4$)-free graph, then $\mathcal{R}_{k+1}(G)$ is connected with diameter at most $4n$. Furthermore, we show that $\mathcal{R}_{k+1}(G)$ is connected for every $k$-colourable ($P_5$, $C_4$)-free graph $G$. 
\end{abstract}

\section{Introduction}

Let $G$ be a finite simple graph with vertex-set $V(G)$ and edge-set $E(G)$. We use $n = |V(G)|$ to denote the number of vertices of $G$ when the context is clear. For a positive integer $k$, a \emph{$k$-colouring} of $G$ is a mapping $\alpha \colon V(G) \to \{1, 2, \ldots, k\}$ such that $\alpha(u) \neq \alpha(v)$ whenever $uv \in E(G)$. We say that $G$ is \emph{$k$-colourable} if it admits a $k$-colouring and the \emph{chromatic number} of $G$, denoted $\chi(G)$, is the smallest integer $k$ such that $G$ is $k$-colourable. The set of vertices that are assigned the same colour is called a \emph{colour class}. We say that the colour classes of a colouring $\alpha$ of $G$ matches the colour classes of a colouring $\beta$ of $G$ if both colourings induce the same colour classes.

The \emph{reconfiguration graph of the $k$-colourings}, denoted $\mathcal{R}_k(G)$, is the graph whose vertices are the $k$-colourings of $G$ and two colourings are joined by an edge if they differ in colour on exactly one vertex. We say that $G$ is \emph{$k$-mixing} if $\mathcal{R}_k(G)$ is connected and the \emph{$k$-recolouring diameter} of $G$ is the diameter of $\mathcal{R}_k(G)$. Given two $k$-colourings $\alpha$ and $\beta$ of $G$, deciding whether there exists a path between the two colourings in $\mathcal{R}_k(G)$ was proved to be PSPACE-complete for all $k >$ 3 \cite{Bonsma2009}. The problem remains PSPACE-complete for graphs with bounded bandwidth and hence bounded treewidth \cite{Wrochna}.

A $k$-colouring of a graph $G$ is called \emph{frozen} if it is an isolated vertex in $\mathcal{R}_k(G)$. In other words, for every vertex $v \in V(G)$, each of the $k$ colours appears in the closed neighbourhood of $v$. One way to show that a graph $G$ is not $k$-mixing is to exhibit a frozen $k$-colouring of $G$. Since every $k$-colouring of $K_k$ is frozen, it is common to study $\mathcal{R}_{k+1}(G)$ for a $k$-colourable graph $G$. For the rest of this paper, we assume that $\ell \ge \chi(G)+1$.

A graph $G$ is $H$-free if no induced subgraph of $G$ is isomorphic to $H$. For a collection of graphs $\mathcal{H}$, $G$ is $\mathcal{H}$-free if $G$ is $H$-free for every $H \in \mathcal{H}$. Let $P_n$, $C_n$, and $K_n$ denote the path, cycle, and complete graph on $n$ vertices, respectively. For two vertex-disjoint graphs $G$ and $H$, the \emph{disjoint union} of $G$ and $H$, denoted by $G+H$, is the graph with vertex-set $V(G) \cup V(H)$ and edge-set $E(G) \cup E(H)$. For a positive integer $r$, we use $rG$ to denote the graph obtained from the disjoint union of $r$ copies of $G$. The goal of this paper is to answer the following question.

\begin{question}
\label{q1}
For which $H$ is every $H$-free graph $\ell$-mixing?
\end{question}

Question \ref{q1} has been answered completely when $H$ is a 3-vertex graph, namely $3K_1$, $P_2+P_1$, $P_3$, and $K_3$. The last author \cite{merkel2021} showed that every $3K_1$-free graph is $\ell$-mixing and the $\ell$-recolouring diameter is at most $4n$. The class of $(P_2+P_1)$-free graphs is a subclass of $P_4$-free graphs, and so we refer to the results on $P_4$-free graphs in this case (see Theorem \ref{thm:p4free}). A graph is $P_3$-free if it is the disjoint union of cliques and it is well known that every $P_3$-free graph is $\ell$-mixing (see e.g. \cite{bonamy2014}). Bonamy and Bousquet \cite{bonamy2018} showed that the $\ell$-recolouring diameter of a $P_3$-free graph is at most $2n$. Cereceda, van den Heuvel, and Johnson \cite{cereceda2008} showed that for all $\ell \ge 3$, there is a bipartite graph that is not $\ell$-mixing (see Figure \ref{fig:frozenbipartite}). Since bipartite graphs are a subclass of $K_3$-free graphs, it follows that not every $K_3$-free graph is $\ell$-mixing. See Table \ref{table:3vertex} for a summary of these results.  

\begin{center}
\begin{table}[H]
\centering
 \begin{tabular}{|c c c|} 
 \hline
 $H$ & Always $\ell$-mixing & Upper bound on the $\ell$-recolouring diameter \\ 
 \hline\hline
 $3K_1$ & YES \cite{merkel2021} & $4n$ \cite{merkel2021} \\ 
 \hline
 $P_2+P_1$ & YES \cite{bonamy2018} & $4n$ \cite{biedl2021} (Theorem \ref{thm:p4free}) \\
 \hline
 $P_3$ & YES \cite{bonamy2014} & $2n$ \cite{bonamy2018} \\
 \hline
 $K_3$ & NO \cite{cereceda2008} & $\infty$ \cite{cereceda2008} (see Figure \ref{fig:frozenbipartite}) \\
 \hline
\end{tabular}
\caption{Summary of recolouring an $H$-free graph for a 3-vertex graph $H$.}
\label{table:3vertex}
\vspace{-8mm}
\end{table}
\end{center}

Next we survey results on Question \ref{q1} when $H$ is a 4-vertex graph. There are 11 graphs on 4 vertices (see Figure \ref{fig:4vertex}). Bonamy and Bousquet \cite{bonamy2018} showed that every $P_4$-free graph $G$ is $\ell$-mixing and the $\ell$-recolouring diameter is at most $2\cdot \chi(G) \cdot n$. Biedl, Lubiw, and Merkel \cite{biedl2021} investigated a class of graphs that generalizes $P_4$-free graphs. The following theorem was not directly stated in \cite{biedl2021} but follows from the proof of Theorem 1 in \cite{biedl2021} and improves the bound on the $\ell$-recolouring diameter of a $P_4$-free graph.

\begin{figure}
    \centering
    \begin{tikzpicture}[scale=1.2]
    \tikzstyle{vertex}=[circle, draw, fill=black, inner sep=0pt, minimum size=5pt]
    \node[vertex](1) at (0,0){};
    \node[vertex](2) at (1,0){};
    \node[vertex](3) at (0,1){};
    \node[vertex](4) at (1,1){};
    \node[] at (0.5, -1){$4K_1$};
    \node[vertex](5) at (0,-2){};
    \node[vertex](6) at (1,-2){};
    \node[vertex](7) at (0,-3){};
    \node[vertex](8) at (1,-3){};
    \node[] at (0.5, -4){$K_4$};
    \draw(5)--(6);
    \draw(5)--(7);
    \draw(5)--(8);
    \draw(6)--(7);
    \draw(6)--(8);
    \draw(7)--(8);
    \node[vertex](9) at (2,0){};
    \node[vertex](10) at (3,0){};
    \node[vertex](11) at (2,1){};
    \node[vertex](12) at (3,1){};
    \draw(9)--(11);
    \node[] at (2.5, -1){co-diamond};
    \node[vertex](13) at (2,-2){};
    \node[vertex](14) at (3,-2){};
    \node[vertex](15) at (2,-3){};
    \node[vertex](16) at (3,-3){};
    \draw(13)--(14);
    \draw(13)--(15);
    \draw(14)--(15);
    \draw(14)--(16);
    \draw(15)--(16);
    \node[] at (2.5, -4){diamond};
    \node[vertex](17) at (4,0){};
    \node[vertex](18) at (4,1){};
    \node[vertex](19) at (5,0){};
    \node[vertex](20) at (5,1){};
    \draw(17)--(18);
    \draw(19)--(20);
    \node[] at (4.5, -1){$2K_2$};
    \node[vertex](21) at (4,-2){};
    \node[vertex](22) at (4,-3){};
    \node[vertex](23) at (5,-2){};
    \node[vertex](24) at (5,-3){};
    \draw(21)--(22);
    \draw(22)--(24);
    \draw(23)--(21);
    \draw(24)--(23);
    \node[] at (4.5, -4){$C_4$};
    \node[vertex](25) at (6,0){};
    \node[vertex](26) at (6,1){};
    \node[vertex](27) at (7,0){};
    \node[vertex](28) at (7,1){};
    \draw(25)--(26);
    \draw(25)--(27);
    \node[] at (6.5, -1){$P_3+P_1$};
    \node[vertex](29) at (6,-2){};
    \node[vertex](30) at (6,-3){};
    \node[vertex](31) at (7,-2){};
    \node[vertex](32) at (7,-3){};
    \draw(29)--(30);
    \draw(29)--(31);
    \draw(30)--(31);
    \draw(29)--(32);
    \node[] at (6.5, -4){paw};
    \node[vertex](33) at (8,0){};
    \node[vertex](34) at (8,1){};
    \node[vertex](35) at (9,0){};
    \node[vertex](36) at (9,1){};
    \draw(33)--(34);
    \draw(33)--(35);
    \draw(33)--(36);
    \node[] at (8.5, -1){claw};
    \node[vertex](37) at (8,-2){};
    \node[vertex](38) at (8,-3){};
    \node[vertex](39) at (9,-2){};
    \node[vertex](40) at (9,-3){};
    \draw(37)--(38);
    \draw(37)--(39);
    \draw(38)--(39);
    \node[] at (8.5, -4){co-claw};
    \node[vertex](41) at (10,-1.5){};
    \node[vertex](42) at (10,-0.5){};
    \node[vertex](43) at (11,-1.5){};
    \node[vertex](44) at (11,-0.5){};
    \draw(41)--(42);
    \draw(41)--(43);
    \draw(44)--(43);
    \node[] at (10.5, -2.5){$P_4$};
    \end{tikzpicture}
    \caption{The 11 graphs on 4 vertices.}
    \label{fig:4vertex}
\end{figure}
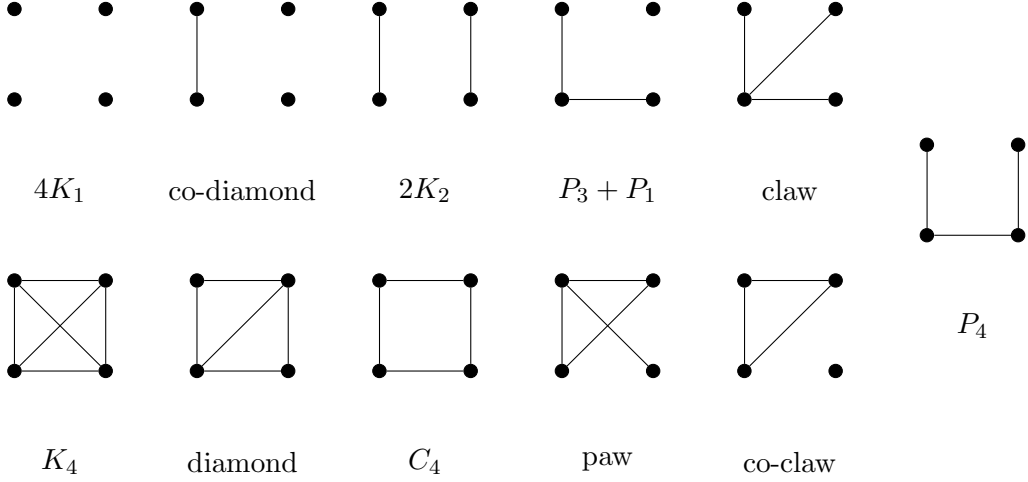

\begin{theorem}[\cite{biedl2021}]
\label{thm:p4free}
The $\ell$-recolouring diameter of a $P_4$-free graph is at most $4n$.
\end{theorem}

Feghali and Merkel \cite{feghali2021} showed that for all $p \ge 1$, there exists a $k$-colourable $2K_2$-free graph with a frozen $(k+p)$-colouring. This result together with the result of Bonamy and Bousquet \cite{bonamy2018} answers Question \ref{q1} when $H$ is a path.

\begin{theorem}[\cite{feghali2021}] \label{thm:p5free}
Every $P_t$-free graph is $\ell$-mixing if and only if $t \le 4$.
\end{theorem}

The connectivity and diameter of $\mathcal{R}_{\ell}(G)$ has also been investigated for classes of graphs defined by more than one forbidden induced subgraph. Bonamy, Johnson, Lignos, Patel, and Paulusma \cite{bonamy2014} showed that every chordal graph and every chordal bipartite graph is $\ell$-mixing and the $\ell$-recolouring diameter is at most $2n^2$. Feghali and Fiala \cite{feghali2020} showed that every co-chordal graph is $\ell$-mixing and the $\ell$-recolouring diameter is at most $2n^2$. They also showed that every 3-colourable ($P_5$, co-$P_5$, $C_5$)-free graph is $\ell$-mixing and the $\ell$-recolouring diameter is at most $2n^2$ \cite{feghali2020}. Biedl, Lubiw, and Merkel \cite{biedl2021} showed that every $P_4$-sparse graph is $\ell$-mixing and the $\ell$-recolouring diameter is at most $4n^2$. The last author \cite{merkel2021} showed that for all $p \ge 1$, there exists a $k$-colourable weakly chordal graph that is not $(k+p)$-mixing.

\subsection*{Our contributions}

We completely answer Question \ref{q1} with the following theorem.

\begin{theorem}
\label{thm:main}
Every $H$-free graph is $\ell$-mixing if and only if $H$ is an induced subgraph of $P_4$ or $P_3+P_1$.
\end{theorem}

We use the following theorems in the proof of Theorem \ref{thm:main}.

\begin{theorem}
\label{thm:4k1c4claw}
For all $p \ge 1$, there exists a $k$-colourable $(4K_1$, $C_4$, claw$)$-free graph that is not $(k+p)$-mixing.
\end{theorem}

\begin{theorem}
\label{thm:k4diamondpawcoclaw}
For all $p \ge 1$, there exists a $k$-colourable $(K_4$, diamond, paw, co-claw, co-diamond$)$-free graph that is not $(k+p)$-mixing.
\end{theorem}

\begin{theorem}
\label{thm:copaw}
Every $(P_3+P_1)$-free graph is $\ell$-mixing and the $\ell$-recolouring diameter is at most $6n$.
\end{theorem}

The proof of Theorem \ref{thm:copaw} leads to a polynomial-time algorithm to find a path of length at most $6n$ between any two $\ell$-colourings of $G$. See Table \ref{table:4vertex} for a summary of these results.

\begin{center}
\begin{table}[H]
\centering
 \begin{tabular}{|c c c|} 
 \hline
 $H$ & Always $\ell$-mixing & Upper bound on the $\ell$-recolouring diameter \\ 
 \hline\hline
 $4K_1$ & NO & $\infty$ (Theorem \ref{thm:4k1c4claw}) \\ 
 \hline
 co-diamond & NO & $\infty$ (Theorem \ref{thm:k4diamondpawcoclaw}) \\ 
 \hline
 $2K_2$ & NO \cite{feghali2021}  & $\infty$ \cite{feghali2021} \\ 
 \hline
 $P_3+P_1$ & YES & $6n$ (Theorem \ref{thm:copaw}) \\ 
 \hline
 claw & NO & $\infty$ (Theorem \ref{thm:4k1c4claw}) \\ 
 \hline
 $P_4$ & YES \cite{bonamy2018} & $4n$ \cite{biedl2021} (Theorem \ref{thm:p4free}) \\
 \hline
 co-claw & NO & $\infty$ (Theorem \ref{thm:k4diamondpawcoclaw}) \\ 
 \hline
 paw & NO & $\infty$ (Theorem \ref{thm:k4diamondpawcoclaw}) \\ 
 \hline
 $C_4$ & NO & $\infty$ (Theorem \ref{thm:4k1c4claw}) \\ 
 \hline
 diamond & NO & $\infty$ (Theorem \ref{thm:k4diamondpawcoclaw}) \\ 
 \hline
 $K_4$ & NO & $\infty$ (Theorem \ref{thm:k4diamondpawcoclaw}) \\ 
 \hline
\end{tabular}
\caption{Summary of recolouring an $H$-free graph for a 4-vertex graph $H$.}
\label{table:4vertex}
\vspace{-8mm}
\end{table}
\end{center}

We also start an investigation into classes of graphs defined by two forbidden induced subgraphs. 

\begin{theorem}
\label{thm:2k2c4}
Every $(2K_2, C_4)$-free graph is $\ell$-mixing and the $\ell$-recolouring diameter is at most $4n$.
\end{theorem}

Note that the class of split graphs, equivalently the class of chordal and co-chordal graphs, is a subclass of ($2K_2$, $C_4$)-free graphs. Theorem \ref{thm:2k2c4} improves the upper bound on the $\ell$-recolouring diameter for a split graph from $2n^2$ to $4n$. We also investigate the superclass of ($P_5$, $C_4$)-free graphs and prove the following.

\begin{theorem}
\label{thm:p5c4}
Every $(P_5, C_4)$-free graph is $\ell$-mixing.
\end{theorem}

We note that there are $P_5$-free graphs and $C_4$-free graphs that are not $\ell$-mixing. 

The rest of the paper is organized as follows. In Section \ref{sec:pre} we give definitions and terminology used throughout the paper. We prove Theorems \ref{thm:4k1c4claw} and \ref{thm:k4diamondpawcoclaw} in Section \ref{sec:frozen} and we prove Theorem \ref{thm:copaw} in Section \ref{sec:copaw}. We prove Theorem \ref{thm:main} in Section \ref{sec:main} and we prove Theorems \ref{thm:2k2c4} and \ref{thm:p5c4} in Section \ref{sec:p5c4}. We end with some open problems in Section \ref{sec:conclusion}.

\section{Preliminaries}
\label{sec:pre}

For a graph $G$, the \emph{complement} of $G$ is the graph with vertex-set $V(G)$ such that the edges of the complement are exactly the non-edges of $G$. A \emph{component} of $G$ is a maximal connected subgraph and an anticomponent of $G$ is a component of the complement of $G$.
For a vertex $v \in V(G)$, the \emph{open neighbourhood} of $v$ is the set of vertices adjacent to $v$ in $G$. The \emph{closed neighbourhood of $v$} is the set of vertices adjacent to $v$ in $G$ together with $v$. For $X,Y \subseteq V(G)$, we say that $X$ is \emph{complete} to $Y$ if every vertex in $X$ is adjacent to every vertex in $Y$. If no vertex of $X$ is adjacent to a vertex of $Y$, we say that $X$ is \emph{anticomplete} to $Y$. Let $G$ and $H$ be vertex-disjoint graphs and let $v \in V(G)$. By \emph{substituting} $H$ for the vertex $v$ of $G$, we mean taking the graph $G-v$ and adding an edge between every vertex of $H$ and every vertex of $G-v$ that is adjacent to $v$ in $G$.

For a colouring $\alpha$ of $G$ and $X \subseteq V(G)$, we say that the colour $c$ \emph{appears} in $X$ if $\alpha(x) = c$ for some $x \in X$, and we use $\alpha(X)$ for the set of colours appearing in $X$ and $|\alpha(X)|$ for the number of colours appearing on $X$. A $k$-colouring of a graph $G$ is called \emph{frozen} if it is an isolated vertex in the recolouring graph $\mathcal{R}_k(G)$. In other words, for every vertex $v \in V(G)$, each of the $k$ colours appears in the closed neighbourhood of $v$.

\section{Frozen colourings}
\label{sec:frozen}

We use frozen colourings to prove Theorems \ref{thm:4k1c4claw} and \ref{thm:k4diamondpawcoclaw}, answering Question \ref{q1} in the negative for 8 out of the 11 cases when $H$ is a 4-vertex graph. To prove Theorem \ref{thm:4k1c4claw}, we construct a family of graphs $\{G_p \mid p \ge 1\}$ such that $G_p$ is a $2p$-colourable ($4K_1$, $C_4$, claw)-free graph that has a frozen $3p$-colouring.

\begin{lemma} \label{lem:Gpfrozen}
For $p \ge 1$, let $G_p$ be the graph obtained from $C_6$ by substituting the complete graph $K_p$ into each vertex. Then $G_p$ is $(4K_1$, $C_4$, claw$)$-free, is $2p$-colourable, and has a frozen $3p$-colouring.
\end{lemma}

\begin{proof}
See Figure \ref{fig:c6blowup} for a $2p$-colouring and frozen $3p$-colouring of $G_p$. Notice that the vertices of $G_p$ can be partitioned into 3 cliques, and so $G_p$ must be $4K_1$-free. Also notice that every vertex of $G_p$ is bisimplicial. That is, the closed neighbourhood of every vertex can be partitioned into two cliques, and so $G_p$ must be claw-free. Finally, suppose $G_p$ contains an induced $C_4$. Then $G_p$ contains an induced $P_3$, call it $P$, with vertices $w, x, y$ in order. Notice that each vertex of $P$ must belong to a distinct copy of $K_p$. But then any other vertex of $G_p$ that is adjacent to both $w$ and $x$ must also be adjacent to $y$, a contradiction.
\end{proof}

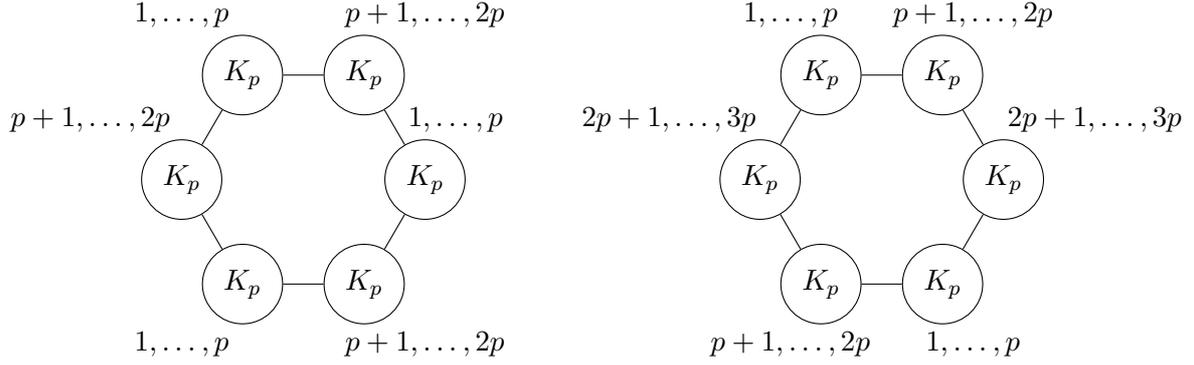
\begin{figure}
    \centering
    \begin{tikzpicture}[scale=0.4]
    \tikzstyle{blowup}=[circle, draw, fill=white, inner sep=0pt, minimum size=30pt]
    \node[blowup](1) at (4,0){$K_p$};
    \node[blowup](2) at (4*cos{60},4*sin{60}){$K_p$};
    \node[blowup](3) at (4*cos{120},4*sin{120}){$K_p$};
    \node[blowup](4) at (4*cos{180},4*sin{180}){$K_p$};
    \node[blowup](5) at (4*cos{240},4*sin{240}){$K_p$};
    \node[blowup](6) at (4*cos{300},4*sin{300}){$K_p$};
    
    \draw(1)--(2);
    \draw(2)--(3);
    \draw(3)--(4);
    \draw(4)--(5);
    \draw(5)--(6);
    \draw(6)--(1);
    
    \node[]() at (4*cos{120}-2,4*sin{120}+2){$1, \ldots, p$};
    \node[]() at (4*cos{60}+2,4*sin{60}+2){$p+1, \ldots, 2p$};
    
    \node[]() at (4*cos{180}-3,4*sin{180}+2){$p+1, \ldots, 2p$};
    \node[]() at (4+1,0+2){$1, \ldots, p$};
    
    \node[]() at (4*cos{240}-2,4*sin{240}-2){$1, \ldots, p$};
    \node[]() at (4*cos{300}+2,4*sin{300}-2){$p+1, \ldots, 2p$};
    
    \end{tikzpicture}
    \hspace{5mm}
    \begin{tikzpicture}[scale=0.4]
    \tikzstyle{blowup}=[circle, draw, fill=white, inner sep=0pt, minimum size=30pt]
    \node[blowup](1) at (4,0){$K_p$};
    \node[blowup](2) at (4*cos{60},4*sin{60}){$K_p$};
    \node[blowup](3) at (4*cos{120},4*sin{120}){$K_p$};
    \node[blowup](4) at (4*cos{180},4*sin{180}){$K_p$};
    \node[blowup](5) at (4*cos{240},4*sin{240}){$K_p$};
    \node[blowup](6) at (4*cos{300},4*sin{300}){$K_p$};
    
    \node[]() at (4*cos{120}-1,4*sin{120}+2){$1, \ldots, p$};
    \node[]() at (4*cos{60}+1,4*sin{60}+2){$p+1, \ldots, 2p$};
    
    \node[]() at (4*cos{180}-3,4*sin{180}+2){$2p+1, \ldots, 3p$};
    \node[]() at (4+3,0+2){$2p+1, \ldots, 3p$};
    
    \node[]() at (4*cos{240}-1,4*sin{240}-2){$p+1, \ldots, 2p$};
    \node[]() at (4*cos{300}+1,4*sin{300}-2){$1, \ldots, p$};
    
    \draw(1)--(2);
    \draw(2)--(3);
    \draw(3)--(4);
    \draw(4)--(5);
    \draw(5)--(6);
    \draw(6)--(1);
    \end{tikzpicture}
    \caption{A $2p$-colouring of $G_p$ and a frozen $3p$-colouring. }
    \label{fig:c6blowup}
\end{figure}

To prove Theorem \ref{thm:k4diamondpawcoclaw}, we use the family of bipartite graphs $\{B_p \mid p \ge 3$\} introduced by Cereceda, van den Heuvel, and Johnson \cite{cereceda2008}. For $p \ge 3$, the graph $B_p$ is 2-colourable, and has a frozen $p$-colouring.

\begin{lemma} \label{lem:Bp}
For $p \ge 3$, let $B_p$ be the graph obtained from the complete bipartite graph $K_{p,p}$ by deleting the edges of a perfect matching. Then $B_p$ is $(K_4$, diamond, paw, co-claw, co-diamond$)$-free, 2-colourable, and has a frozen $p$-colouring.
\end{lemma}

\begin{proof}
See Figure \ref{fig:frozenbipartite} for a frozen $p$-colouring of $B_p$. Since $B_p$ is bipartite, it is also $K_3$-free and so it must be ($K_4$, diamond, paw, co-claw)-free. To show that $B_p$ is co-diamond-free, we show that the complement of $B_p$ is diamond-free. The complement of $B_p$ consists of two disjoint cliques with $p$ vertices, call them $Q_1$ and $Q_2$, and the edges between $Q_1$ and $Q_2$ are a perfect matching. Suppose by contradiction that the complement of $B_p$ contains an induced diamond $D$. Let $w,x$ be the adjacent vertices of degree 3 in $D$ and let $y,z$ be the non-adjacent vertices of degree 2 in $D$. Since each vertex in $Q_i$, for $i$ = 1, 2, can be adjacent to at most one vertex in the other clique, either $x$, $w$, $y \in Q_1$ or $x$, $w$, $z \in Q_1$. Without loss of generality, let $x$, $w$, $y \in Q_1$. Since $z$ is not adjacent to $y$, $z$ must be in $Q_2$. But then $z \in Q_2$ must be adjacent to two vertices in $Q_1$, a contradiction. 
\end{proof}

\begin{figure}
\centering
\begin{tikzpicture}[scale=0.5]
\tikzstyle{vertex}=[circle, draw, fill=black, inner sep=0pt, minimum size=5pt]

    \node[vertex, label=left:1](1) at (0,4) {};
    \node[vertex, label=left:2](2) at (0,2) {};
    \node[vertex, label=left:3](3) at (0,0) {};
	\node[vertex, label=left:$p$](4) at (0,-4) {};
	\node[vertex, label=right:1](5) at (4,4) {};
	\node[vertex, label=right:2](6) at (4,2) {};
	\node[vertex, label=right:3](7) at (4,0) {};
	\node[vertex, label=right:$p$](8) at (4,-4) {};
	
	\node at (0,-2) {\textbf{\vdots}};
	\node at (4,-2) {\textbf{\vdots}};
     
    \draw(1)--(6);
    \draw(1)--(7);
    \draw(1)--(8);
    \draw(2)--(5);
    \draw(2)--(7);
    \draw(2)--(8);
    \draw(3)--(5);
    \draw(3)--(6);
    \draw(3)--(8);
    \draw(4)--(5);
    \draw(4)--(6);
    \draw(4)--(7);
    
\end{tikzpicture}
\caption{A frozen $p$-colouring of the graph $B_p$. \cite{cereceda2008}}
\label{fig:frozenbipartite}
\end{figure}
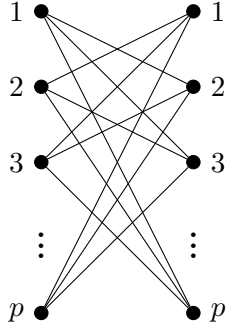

\section{Recolouring $(P_3+P_1)$-free graphs}
\label{sec:copaw}
In this section, we investigate the connectivity and diameter of $\mathcal{R}_{\ell}(G)$ for a $(P_3+P_1)$-free graph $G$. Our strategy takes advantage of the anticomponents of $G$. Since there are all possible edges between the anticomponents of a graph, in any colouring of the graph, the colours that appear on each anticomponent must be pairwise distinct. Thus if we can recolour a vertex in some anticomponent of $G$, that recolouring step would extend to a recolouring step for all of $G$.

Given two arbitrary $\ell$-colourings $\alpha$ and $\beta$ of $G$, we show how to recolour $\alpha$ into $\beta$ using the following strategy. First fix an arbitrary $\chi(G)$-colouring $\gamma$ of $G$ by finding an optimal colouring of each anticomponent of $G$. Next recolour $\alpha$ into a $\chi(G)$-colouring $\alpha'$ of $G$ by recolouring each anticomponent $A$ of $G$ so that the colour classes of $\alpha'_A$ match the colour classes of $\gamma_A$. Similarly recolour $\beta$ into a $\chi(G)$-colouring $\beta'$ of $G$ by recolouring each anticomponent $A$ of $G$ so that the colour classes of $\beta'_A$ match the colour classes of $\gamma_A$. Once we have the colourings $\alpha'$ and $\beta'$, we use the following Renaming Lemma.

\begin{lemma}[Renaming Lemma \cite{bonamy2018}]
\label{lem:recolour}
Let $\alpha'$ and $\beta'$ be two $k$-colourings of $G$ that induce the same partition of vertices into colour classes and let $\ell \ge k+1$. Then $\alpha'$ can be recoloured into $\beta'$ in $\mathcal{R}_{\ell}(G)$ by recolouring each vertex at most 2 times.
\end{lemma}

Olariu \cite{olariu1988} proved the following useful result characterizing the anticomponents of a $(P_3+P_1)$-free graph.

\begin{theorem}[\cite{olariu1988}]
\label{pawfree}
Each connected component of a paw-free graph is either $K_3$-free or $(P_2+P_1)$-free.
\end{theorem}

Therefore each anticomponent of a $(P_3+P_1)$-free graph is either $3K_1$-free or $P_3$-free. The last author proved the following theorem on recolouring $3K_1$-free graphs.

\begin{theorem}[\cite{merkel2021}]
\label{3k1free}
Every $3K_1$-free graph is $\ell$-mixing and the $\ell$-recolouring diameter is at most $4n$.
\end{theorem}

The proof of Theorem \ref{3k1free} uses the following lemma which we will use in the proof of Theorem \ref{thm:copaw}. 

\begin{lemma}[\cite{merkel2021}]
\label{lem:3k1}
Let $\gamma$ be a $\chi(G)$-colouring of a $3K_1$-free graph $G$. Any $\ell$-colouring of $G$ can be recoloured into a $\chi(G)$-colouring $\gamma'$ such that the colour classes of $\gamma'$ match the colour classes of $\gamma$ by recolouring each vertex at most once.
\end{lemma}

We are now ready to prove Theorem \ref{thm:copaw}.

\begin{proof}[Proof of Theorem \ref{thm:copaw}]
Let $G$ be a $k$-colourable $(P_3+P_1)$-free graph and let $\alpha$ and $\beta$ be two $\ell$-colourings of $G$ with $\ell \ge k+1$. Fix a $\chi(G)$-colouring $\gamma$ of $G$. We note that this colouring can be found in polynomial time by optimally colouring each anticomponent of $G$ \cite{kral2001}.

Let $A_1, A_2, \ldots, A_p$ be the anticomponents of $G$. 
Since the anticomponents of any graph are pairwise complete, the colours used on each anticomponent are pairwise distinct, and thus $\chi(G) = \sum_{1 \le i \le p}\chi(A_i)$. We proceed by recolouring each anticomponent $A_i$ of $G$ so that the colouring of $A_i$ matches the colour classes of $A_i$ in $\gamma$. The anticomponent that is selected to be recoloured is determined based on the following claim. 

\begin{claim}
\label{claim1}
In any $\ell$-colouring $\alpha$ of $G$, either some colour does not appear in $\alpha$ or there is some anticomponent $A$ of $G$ such that $\alpha$ uses at least $\chi(A)$+1 colours on $A$.
\end{claim}
\textit{Proof of claim}: 
Suppose $\alpha$ colours each anticomponent $A_i$ of $G$ with $\chi(A_i)$ colours. Since 
$\chi(G) = \sum_{1 \le i \le p}\chi(A_i)$, only $\chi(G)$ colours appear in $\alpha$ and since $\ell \ge \chi(G)+1$, some colour does not appear in $\alpha$. 
This completes the proof of the Claim 1.\\

If some colour $c$ is not being used, we take some anticomponent $A_i$ of $G$ that has not been recoloured. Recolour $A_i$ as described below depending on whether $A_i$ is $P_3$-free or $3K_1$-free using $c$ together with the $\chi(A_i)$ colours appearing on $A_i$.  

Now suppose that there is some anticomponent $A_i$ where the current colouring uses at least $\chi(A_i)+1$ colours on $A_i$. If $A_i$ is $P_3$-free then since each component of $A_i$ is a clique, we use the Renaming Lemma to recolour each component of $A_i$ to match the colour classes of $\gamma$ by recolouring each vertex at most twice. If $A_i$ is $3K_1$-free, then we use Lemma \ref{lem:3k1} to recolour the current colouring of $A_i$ to match the colour classes of $\gamma$.

We have recoloured $\alpha$ into a $\chi(G)$-colouring $\alpha'$ whose colour classes correspond to the colour classes of $\gamma$ by recolouring each vertex at most twice. Similarly, we recolour $\beta$ into a $\chi(G)$-colouring $\beta'$ whose colour classes correspond to those of $\gamma$ by recolouring each vertex at most twice. By the Renaming Lemma, $\alpha'$ can be recoloured into $\beta'$ by recolouring each vertex at most twice. Thus, we have found a path from $\alpha$ to $\beta$ in $\mathcal{R}_\ell(G)$ by recolouring each vertex at most 6 times.
\end{proof}

\section{Proof of Theorem \ref{thm:main}}
\label{sec:main}

Theorem \ref{thm:main} follows from the following lemma.

\begin{lemma} \label{lem:No5} 
There is no graph $H$ on five vertices such that every $H$-free graph $G$ is  ($\chi(G)$+1)-mixing.
\end{lemma}

\begin{proof}
There are 34 graphs on 5 vertices. We partition them into four groups: Group 1 consists of the 21 graphs which are not bipartite. Group 2 consists of the 7 graphs which are bipartite and contain the co-diamond. Group 3 consists of $P_5$. Group 4 consists of the remaining 5 graphs, which are $5K_1$, $K_{1,4}$, claw$+K_1$, $K_{2,3}$ and the banner, where the banner consists of the cycle $C_4$ on four vertices together with a vertex joined to exactly one vertex of the $C_4$.

The graph $B_p$ is bipartite and, as mentioned in Lemma \ref{lem:Bp}, is    co-diamond-free, and thus is $H$-free for any non-bipartite graph $H$ and for any graph $H$ containing the co-diamond. By Lemma \ref{lem:Bp}, $B_p$ has a frozen $p$-colouring, for every $p \ge 3$.

Theorem \ref{thm:p5free} \cite{feghali2021} says that not every $P_5$-free graph is ($\chi$+1)-mixing.

Lemma \ref{lem:Gpfrozen} says that there are frozen colourings of the graph $G_p$ which is $(4K_1$, $C_4$, claw$)$-free, and thus also ($5K_1$, $K_{1,4}$, claw$+K_1$, $K_{2,3}$, banner)-free. Thus, for every  graph $H$ in Group 4, there is an $H$-free graph which is not ($\chi$+1)-mixing. 




\end{proof}

\section{Recolouring $(P_5, C_4)$-free graphs}
\label{sec:p5c4}

A \textit{blow-up} of a graph $G$ with vertices $v_1$, $v_2, \dots, v_{n}$ is any graph $H$ that can be obtained by substituting a clique $A_i$ for $v_i$ in $G$, one at a time for all $i \in [n]$. A \textit{clique cutset} $Q$ of a graph $G$ is a clique in $G$ such that $G$-$Q$ has more components than $G$. 

A \textit{split graph} is one whose vertex-set can be partitioned into a clique and an independent set. Equivalently, split graphs are the class of  $(2K_2,C_4,C_5)$-free graphs or the class of graphs which are both chordal and cochordal. Bl\'{a}zsik et al \cite{zoltan1993} gave the following characterization of $(2K_2,C_4)$-free graphs. 

\begin{theorem}[\cite{zoltan1993}]
\label{2K2C4free}
A graph $G$ is $(2K_2,C_4)$-free if and only if its vertex-set can be partitioned into three possibly empty sets, $Q,C$ and $I$ such that $Q$ is a complete graph, $I$ is an independent set and $C$ (if nonempty) induces a $C_5$, and $C$ is complete to $Q$ and anticomplete to $I$.
\end{theorem}

First we consider the subclass $(2K_2,C_4)$-free graphs and prove Theorem \ref{thm:2k2c4}.

\begin{proof}[Proof of Theorem \ref{thm:2k2c4}]
Let $\gamma_{1}$ and $\gamma_{2}$ be any two colourings of a $(2K_2,C_4)$-free graph $G$. We prove the theorem by providing a path between the two colourings of length at most $4n$. First we obtain two colourings $\alpha$ and $\beta$ from $\gamma_{1}$ and $\gamma_{2}$, respectively, such that $\alpha$ and $\beta$ induce the same partition of the vertices into colour classes. We consider two cases. \\
\textit{Case 1}: When $G$ has an induced $C_5$. By 
Theorem \ref{2K2C4free}, $V(G)$ can be partitioned into three subsets $C$, $Q$, and $I$, where $C$ induces a $C_5$, $Q$ is a clique and $I$ is an independent set. Since $C$ is complete to $Q$ and anticomplete to $I$, we have $\chi (G)$ = 3+$|Q|$ and $\gamma_{i}(C) \cap \gamma_{i}(Q)$ = $\emptyset$, for $i$ = 1, 2. Let $C$ = $\{v_1,v_2, \dots,v_5\}$ such that $v_jv_{j+1} \in E(G)$ for all $j$ ($mod\ 5$). Since $\ell \geq \chi(G)+1=4+|Q|$, if $\gamma_i$ uses three colours on $C$ then there is a colour that does not appear on $C \cup Q$ and we recolour a vertex of $C$ with this colour to use at least four colours on $C$. There exist three vertices, say $v_1$, $v_2$, $v_3$, such that $\gamma_{i}(v_1)\neq \gamma_{i}(v_2) \neq \gamma_{i}(v_3)$ and $\gamma_{i}(v_1)\neq \gamma_{i}(v_3)$, for $i$ = 1, 2. We define colourings $\alpha$ and $\beta$, from $\gamma_{1}$ and $\gamma_{2}$, respectively, as follows;

\begin{equation*}
\alpha(v) = 
    \begin{cases}
                                   \gamma_1(v_1) & \textit{if } v = v_4 \\
                                   \gamma_1(v_2) & \textit{if } v = v_5 \\
                                   \gamma_1(v_1) & \textit{if } v\in I \\           \gamma_1(v) & \textit{if } v\in \{v_1,v_2,v_3\} \cup Q ;
    \end{cases}
\hspace{10mm}
\beta(v) = 
\begin{cases}
                                   \gamma_2(v_1) & \textit{if } v = v_4 \\
                                   \gamma_2(v_2) & \textit{if } v = v_5 \\
                                   \gamma_2(v_1) & \textit{if } v\in I \\                      \gamma_2(v) & \textit{if } v\in \{v_1,v_2,v_3\} \cup Q ;
  \end{cases}
\end{equation*}
\textit{Case 2}: When $G$ is $C_5$-free. Since $G$ is $(2K_2,C_4,C_5)$-free, $G$ is a split graph. Hence, $V(G)$ can be partitioned into subsets $Q$ and $I$, where $Q$ is a maximum clique and $I$ is an independent set. Let $Q$ = $\{u_1,u_2, \dots,u_p\}$. We partition $I$ into subsets $I_1$, $I_2, \dots, I_p$, where $I_j$ = $\{v\in I|\ v$ is non-adjacent to $u_j$ and adjacent to $u_1, \dots,u_{j-1}\}$, for all $j\in [p]$. We define colourings $\alpha$ and $\beta$, from $\gamma_{1}$ and $\gamma_{2}$, respectively, as follows:

\begin{equation*}
  \alpha(v) =
  \begin{cases}
                                   \gamma_1(v_j) & \textit{if } v \in I_j \\
                                   \gamma_1(v) & \textit{if } v\in Q ;
  \end{cases}
\hspace{10mm}
  \beta(v) =
  \begin{cases}
                                   \gamma_2(v_j) & \textit{if } v \in I_j \\
                                   \gamma_2(v) & \textit{if } v\in Q ;
  \end{cases}
\end{equation*}

Therefore, in both cases, we obtain $\alpha$ and $\beta$ from $\gamma_{1}$ and $\gamma_{2}$, respectively, by recolouring a vertex at each step in at most $n$ steps. Since $\alpha$ and $\beta$ induce the same partition of the vertices into colour classes and use $\chi(G)$ colours, by the Renaming Lemma there exists a path between the colourings of length at most $2n$ in $R_\ell(G)$. Hence there exists a path between $\gamma_{1}$ and $\gamma_{2}$ in $R_\ell(G)$ of length at most $4n$.
\end{proof}

A clique cutset $Q$ in a graph $G$ is called a \textit{tight clique cutset} if there exists a component $H$ of $G$-$Q$ which is complete to Q, and then $H$ is called a \textit{tight component}. 

\begin{lemma}\label{Hiso}
    Suppose $G$ is a graph with a tight clique cutset $Q$ and let $H$ be a tight component of $G$-$Q$. Let $V_2$ = $V(G) \setminus V(H)$. Let $\alpha$ and $\beta$ be two $\ell$-colourings of $G$ such that $\alpha_{V_2}$ = $\beta_{V_2}$. Suppose $H$ is $\ell_1$-mixing for all $\ell_1 \geq \chi(H)+1$. Then there exists a path between $\alpha$ and $\beta$ in $R_{\ell}(G)$.
\end{lemma}
\begin{proof}
    Let $V_1$ = $V(H)$ and $V_2$ = $V(G)\setminus V(H)$. Since $\ell > \chi(G) \ge \chi(H) + |Q|$ and $\alpha(Q) = \beta(Q)$, there exist at least $\ell_1$ colours that do not appear in $\alpha(Q)$. Let $A$ be the set of all colours that does not appear in $\alpha(Q)$.

    Since $H$ is $\ell_1$-mixing, there exists a path between $\alpha_1$, the restriction of $\alpha$ on $V_1$, and $\beta_1$, the restriction of $\beta$ on $V_1$, in $R_{\ell_1}(H)$ such that every colouring in the path uses only colours in $A$. We know that none of the colours in $A$ appear in $\alpha(Q)$, so we can extend every colouring in the path by colouring every vertex $v \in V_2$ with the colour $\alpha(v)$ = $\beta(v)$.
\end{proof}

\begin{lemma}\label{cutlemma}
    Suppose $G$ has a tight clique cutset $Q$. If every induced subgraph $F$ of $G$ is $\chi(F)+c$-mixing, for all $c\geq 1$, then $G$ is $\ell$-mixing.
\end{lemma}
\begin{proof}
    Suppose $G$ has a tight clique cutset $Q$ and $H$ is a tight component of $G$-$Q$. Let $V_1$ = $V(H)$, $V_2$ = $V(G) \setminus V_1$, $G_2$ = $G[V_2]$, and $\ell_2 \geq \chi(G_2)+1$.

    Let $\alpha$ be an $\ell$-colouring of $G$ and let $\gamma$ be a $\chi$-colouring of $G$. We shall prove that there exists a path between $\alpha$ and $\gamma$ in $R_{\ell}(G)$. Since $G_2$ is $\ell_2$-mixing, there exists a path, say $P$, between $\alpha_2$, the restriction of $\alpha$ on $V_2$, and $\gamma_2$, the restriction of $\gamma$ on $V_2$, in $R_{\ell_2}(G_2)$.

    Let $\epsilon_2$ and $\psi_2$ be any two colourings adjacent in the path $P$. Note that any $\ell$-colouring of $G_2$ can be extended to an $\ell$-colouring of $G$.

    \begin{claim}
        There exists a path between any two $\ell$-colourings $\epsilon$ and $\psi$ of $G$, whose restrictions on $V_2$ are $\epsilon_2$ and $\psi_2$, respectively.
    \end{claim}

    By Lemma \ref{Hiso}, it is sufficient to prove that there exists a path between $\epsilon$ and some colouring $\psi$ whose restriction on $V_2$ is $\psi_2$. We may also assume that $\epsilon_1$, the restriction of $\epsilon$ on $V_1$, uses $\chi(H)$ colours.

    Since $\epsilon_2$ and $\psi_2$ are adjacent in the path $P$, there is a unique vertex $v \in V_2$ such that $\epsilon_{2}(v) \neq \psi_{2}(v)$. If $v \notin Q$ or $\psi(v)\notin \epsilon(V_1)$, we can recolour $v$ in $\epsilon$ with the colour $\psi(v)$ to obtain $\psi$.

    Let $v\in Q$ and $\psi(v)\in \epsilon(V_1)$. Since $\ell > \chi(H)+|Q|$, there exists a colour, say $r$, that does not appear in $\epsilon(V_1\cup Q)$. Starting with $\epsilon$ recolour every vertex in $V_1$ coloured $\psi(v)$ with the colour $r$, then recolour $v$ with the colour $\psi(v)$ to obtain the colouring $\psi$.
\end{proof}

\begin{proof}[Proof of Theorem \ref{thm:p5c4}] 
Let $G$ be a connected $(P_5, C_4)$-free graph. The proof is by induction on the number of vertices. If there is no induced $C_5$ in $G$, then $G$ is a chordal graph and hence $R_\ell(G)$ is connected \cite{bonamy2014}. If there is an induced $C_5$ in $G$, then by the structure of $(P_5, C_4)$-free graphs given in \cite{fouquet1995} there exists a blowup of $C_5$, say $C$, and a clique $Q$ such that $C$ is complete to $Q$ and $C$ is a component of $G$-$Q$. Hence the proof follows from Lemma \ref{cutlemma}.
\end{proof}

Note that there are other classes of graphs that admit a tight clique cutset, for example, chordal graphs.
\section{Conclusion}
\label{sec:conclusion}

We have proved a dichotomy theorem: The reconfiguration graph of the ($k$+1)-colourings of a graph $G$, $R_{k+1}(G)$, is connected for every $k$-colourable $H$-free graph $G$ if and only if $H$ is an induced subgraph of $P_4$ or $P_3+P_1$. It is interesting to note that these are exactly the family of $H$-free graphs for which the colouring problem is polynomial-time solvable \cite{kral2001}. Next it would be natural to do try to do the same for $(H_1, H_2)$-free graphs.

One might first consider graphs on four vertices. There are 11 such graphs, however, by the dichotomy result, we do not need to consider $P_4$ or $P_3+P_1$. Of the remaining 36 pairs of the other 9 graphs on four vertices, the recolouring diameter is infinite for 3 pairs by Theorem 4 and for 10 other pairs by Theorem 5, and finite for $(2K_2, C_4)$-free graphs. Further, it is easy to see that $C_6$ has a frozen 3-colouring. The only graphs on four vertices which are induced subgraphs of $C_6$ are $P_4$, $P_3+P_1$ and $2K_2$. Thus it is of particular interest to study the connectivity of the reconfiguration graph of the $k$-colourings of $(2K_2, H)$-free graphs, where $H$ is a graph with at least four vertices and is not $P_4$, $P_3+P_1$ or $C_4$.

Theorem 8 states that the class of $(P_5, C_4)$-free graphs is $\ell$-mixing, however the $\ell$-recolouring diameter is still unknown. It is interesting to find other classes of $(H_1, H_2)$-free graphs which are $\ell$-mixing and the $\ell$-recolouring diameter is linear or polynomial. In \cite{feghali2021}, Feghali and Merkel asked if there is a $k$-colourable graph $G$ such that $\mathcal{R}_{k+1}(G)$ is disconnected but every component of $\mathcal{R}_{k+1}(G)$ has at least two vertices. When $k$ = 2, L. Cereceda \textit{et.al} \cite{cereceda2008} proved that there exist bipartite graphs that are not 3-mixing and do not admit a frozen 3-colouring, for example $C_8$. Furthermore, let $H$ be any graph with $V(H)$ = $\{v_1, v_2, \dots, v_n\}$ which admits a frozen $\ell$-colouring, $\ell\geq 3$. We can construct a connected graph $G$ from $H$ by adding vertices $u_1, \dots, u_i$, $1 < i\leq n$, and edges $v_{j}u_{j}$ for all $j\in [i]$. Note that $\chi(G)$ = $\chi(H)$ and $G$ contains $H$ as an induced subgraph, but $G$ is not $\ell$-mixing and does not admit a frozen $\ell$-colouring. We ask the following: 
\begin{question}
Is there a graph $G$ which is not $\ell$-mixing, $\ell\geq 4$, and does not contain an induced subgraph which admits a frozen $\ell$-colouring?
\end{question}




\end{document}